\documentclass[12pt]{amsart}
\usepackage{amsmath,amssymb}
\usepackage{amsfonts}
\usepackage{amsthm}
\usepackage{latexsym}
\usepackage{graphicx}
\def\p{\partial}

\def\R{\mathbb{R}}

\def\vv<#1>{\langle#1\rangle}
\def\sph{\mathbb{S}}

\def\XXint#1#2{\setbox0=\hbox{$#1{#2}{\int}$}{#2}\kern-.5\wd0 }

\def\XXint#1#2#3{{\setbox0=\hbox{$#1{#2#3}{\int}$}
     \vcenter{\hbox{$#2#3$}}\kern-.5\wd0}}

\def\vv<#1>{\langle#1\rangle}

\def\sn{{\rm sn}}
\def\cs{{\rm cs}}

\def\wt{\widetilde}
\def\cut{{\rm cut}}
\newtheorem{thm}{Theorem}[section]

\newtheorem{cor}{Corollary}[section]
\theoremstyle{definition}

\theoremstyle{remark}

\numberwithin{equation}{section}

\begin{document}

\title{Some aspects of  Cheng-Yau gradient estimates}
\author{Qixuan Hu $^1$, Chengjie Yu$^2$}
\address{Qixuan Hu\\ Department of Mathematics\\Shantou University\\P. R. China}
\email{qxhu@stu.edu.cn}
\address{Chengjie Yu\\ Department of Mathematics\\Shantou University\\P. R. China}
\email{cjyu@stu.edu.cn}
\begin{abstract}
In this note, we extend the rigidity of Cheng-Yau gradient estimate in \cite{HXY} to surfaces with lower Ricci curvature bound.  Motivated by these  sharp Cheng-Yau gradient estimates,  pointwise Cheng-Yau gradient estimates for higher dimensional Riemannian manifolds are obtained, and as their applications,  monotonicity formulas for positive harmonic functions are obtained.
\end{abstract}
\thanks{$^1$ Research partially supported by SRIG from Shantou University with contract no. NTF25026T}
\thanks{$^2$ Research partially supported by GDNSF with contract no. 2025A1515011144.}
\renewcommand{\subjclassname}{%
  \textup{2020} Mathematics Subject Classification}
\date{}
\maketitle
\section{Introduction}
In their seminal work \cite{CY} (see also \cite{Li,SY}), Cheng and Yau established  the following local gradient estimate of positive harmonic functions:
\begin{equation}\label{eq-CY}
\sup_{B_p(R)}\|\nabla \log u\|\leq C_n\left(\frac{1+\sqrt K R}{R}\right)
\end{equation}
where $u$ is a positive harmonic function on $B_p(2R)$ in a complete Riemannian manifold $(M^n,g)$ with Ricci curvature bounded from below by $-(n-1)K$ on $B_p(2R)$ for some nonnegative constant $K$. This estimate plays an important role in geometric analysis and is now a fundamental tool in the field which is called the Cheng-Yau gradient estimate. For example, it implies the classical  Harnack inequality on Riemannian manifolds and  the Liouville theorem for harmonic functions on complete Riemannian manifolds with nonnegative Ricci curvature directly.

The Cheng-Yau gradient estimate \eqref{eq-CY} was later refined by Li-Wang \cite{LW} to the following form:
\begin{equation}\label{eq-LW}
\sup_{B_p(R)}\|\nabla\log u\|\leq (n-1)\sqrt{K}+\frac{C_n}{R}.
\end{equation}
This estimate is sharp in leading term globally. The lower order term was refined to be sharp globally by Munteanu \cite{Mu}:
\begin{equation}\label{eq-Mu}
\sup_{B_p(R)}\|\nabla\log u\|\leq (n-1)\sqrt{K}+\frac{C_n}{R}\exp\left(-C\sqrt K R\right)
\end{equation}
where $C$ is a universal positive constant.

The estimate \eqref{eq-Mu} is only asymptotically sharp when $R\to +\infty$ and is not sharp locally. In the interesting work \cite{Xu}, Xu considered the problem of Cheng-Yau gradient estimate that is sharp locally. Xu first considered the problem on $\R^n$ and obtained the following result:
\begin{thm}[Xu \cite{Xu}]\label{thm-Xu} Let $B(R)$ be the ball of radius $R$ centered at the origin of $\R^n$ with $n\geq 2$ and $u$ be a positive harmonic function on $B(R)$. Then,
\begin{equation}\label{eq-CY-Eu}
 \|\nabla \ln u\|(x)\leq \frac{n-1}{R-\|x\|}+\frac{1}{R+\|x\|},\ \forall x\in B(R).
\end{equation}
Moreover, the equality of \eqref{eq-CY-Eu} holds for some $x\in B(R)$ if and only if $$u(\cdot)=\lambda P(\cdot,y)$$ for some $y\in \p B(R)$ and $\lambda >0$. Here
\begin{equation}\label{eq-Poisson}
 P(x,y)=\frac{R^2-|x|^2}{n\omega_nR|y-x|^n}
\end{equation}
is the Poisson kernel of $B(R)$.
\end{thm}

It is then a natural question that if the estimate \eqref{eq-CY-Eu} also holds on complete Riemannian manifolds with nonnegative Ricci curvature. In \cite{Xu}, Xu also gave an affirmative answer to this question when $n=2$.
\begin{thm}[Xu \cite{Xu}]\label{thm-Xu-surface} Let $(M^2,g)$ be a complete Riemannian surface, $p\in M$. and $u$ be a positive harmonic function on $B_p(R)$. Suppose the Ricci curvature is nonnegative on $B_p(R)$. Then,
\begin{equation}\label{eq-Xu}
\|\nabla\log u\|(x)\leq \frac{2R}{R^2-r^2(x)}\ \forall\ x\in B_p(R),
\end{equation}
where $r(x)=d(p,x)$ is the distance between $p$ and $x$.
\end{thm}
Although the estimate \eqref{eq-Xu} is sharp on the Euclidean plane, it is unable to characterize its rigidity in \cite{Xu} because \eqref{eq-Xu} was proved in \cite{Xu} by a complicated cut-off argument in similar spirits as the proofs of \eqref{eq-CY}, \eqref{eq-LW} and \eqref{eq-Mu}. In \cite{HXY}, joint with Xu,  we gave a new proof to \eqref{eq-Xu} by maximum principle and characterized the rigidity of \eqref{eq-Xu}.
\begin{thm}[Hu-Xu-Yu \cite{HXY}]\label{thm-HXY}Let the notations be the same as in Theorem \ref{thm-Xu-surface}. Then, the equality of \eqref{eq-Xu} holds for some point in $B_p(R)$ if and only if $B_p(R)$ is isometric to a ball of radius $R$ in $\R^2$ and $u$ is a positive multiple of the Poisson kernel w.r.t. some point in $\p B_p(R)$.
\end{thm}

In this paper, we consider the problem of Cheng-Yau gradient estimate that is sharp locally for Riemannian manifolds with a Ricci curvature lower bound.  By using conformal transformations, we are able to extend the results of Theorem \ref{thm-Xu} to space forms. However, because the Laplacian operator is not conformal invariant when $n\geq 3$, we only get sharp Cheng-Yau gradient estimates for positive solutions to the conformal Laplacian equation when $n\geq 3$.

Before stating the main results, we first fix some notations. For a constant $K$, we define
$$\sn_K(r)=\left\{\begin{array}{cl}\frac{\sin{(\sqrt Kr)}}{\sqrt K}&K>0\\
r&K=0\\
\frac{\sinh(\sqrt{-K}r)}{\sqrt{-K}}&K<0,
\end{array}\right.$$
and
$$\cs_K(r)=\sn'_K(r)=\left\{\begin{array}{cl}\cos(\sqrt K r)&K>0\\
1&K=0\\
\cosh(\sqrt{-K}r)&K<0.
\end{array}\right.$$
\begin{thm}\label{thm-CY-conf}
Let $\mathbb{M}^n_K$ be the space form with constant sectional curvature $K$, $p\in \mathbb{M}^n_K$, and $u$ be a positive solution to the conformal Laplacian equation on $\mathbb{M}^n_K$:
$$\Box_{{\mathbb{M}^n_K}}u:=-\Delta_{\mathbb{M}^n_K} u+\frac{n(n-2)K}{4}u=0$$ on $B_p(R)$ with $R<\frac{\pi}{\sqrt K}$ when $K>0$. Then,
\begin{equation}\label{eq-CY-conf}
\begin{split}
\left\|\nabla\ln \left(\cs_K^{{n-2}}\left(\frac{r}2\right)u\right)\right\|(x)\leq \frac{\cs_K\left(\frac{R}{2}\right)}{\cs_K\left(\frac{r(x)}{2}\right)}\left(\frac{n-1}{2\sn_K\left(\frac{R-r(x)}{2}\right)}+\frac{1}{2\sn_K\left(\frac{R+r(x)}{2}\right)}\right)
\end{split}
\end{equation}
for any $x\in B_p(R)$, where $r(x)=d(p,x)$. Moreover, the equality of \eqref{eq-CY-conf} holds for some point in $B_p(R)$ if and only if $u$ is a positive multiple of the Poisson kernel of $\Box_{\mathbb{M}_K^n}$ for $B_p(R)$ w.r.t. some point in $\p B_p(R)$. In particular, when $n=2$,
\begin{equation}\label{eq-CY-2D}
\|\nabla \ln u\|(x)\leq \frac{\sn_K(R)}{2\sn_K\left(\frac{R+r(x)}{2}\right)\sn_K\left(\frac{R-r(x)}{2}\right)},\ \forall x\in B_p(R).
\end{equation}
Moreover, the equality of \eqref{eq-CY-2D} holds for some point in $B_p(R)$ if and only if $u$ is a positive multiple of the Poisson kernel for $B_p(R)$ w.r.t. some point in $\p B_p(R)$.
\end{thm}

Motivated by \eqref{eq-CY-2D}, we obtain Cheng-Yau gradient estimate that is locally sharp for complete Riemannian surfaces with curvature lower bound and characterize its rigidity.
\begin{thm}\label{thm-CY-surface}
Let $(M^2,g)$ be a complete Riemannian surface, $p\in M$ and $u$ be a positive harmonic function on $B_p(R)$. Suppose that the Gaussian curvature on $B_p(R)$ is no less than $K$ and  $R<\frac{\pi}{\sqrt K}$ when $K>0$. Then,
\begin{equation}\label{eq-CY-surface}
\|\nabla \ln u\|(x)\leq \frac{\sn_K(R)}{2\sn_K\left(\frac{R+r(x)}{2}\right)\sn_K\left(\frac{R-r(x)}{2}\right)}\ \forall x\in B_p(R),
\end{equation}
where $r(x)=d(p,x)$. Moreover, the equality of \eqref{eq-CY-surface} holds for some point in $B_p(R)$ if and only if $B_p(R)$ is isometric to a geodesic ball of radius $R$ in $\mathbb{M}_K^2$ and $u$ is a positive multiple of the Poisson kernel for $B_p(R)$ w.r.t. some point in $\p B_p(R)$.
\end{thm}

Furthermore, motivated by Theorem \ref{thm-CY-surface}, we obtain the following pointwise local Cheng-Yau gradient estimate for higher dimensional complete Riemannian manifolds with Ricci curvature bounded from below.

\begin{thm}\label{thm-CY-manifold}
Let $(M^n,g)$ be a complete Riemannian manifold with  $n\geq 3$, $p\in M$ and $u$ be a positive harmonic function on $B_p(R)$.  Suppose that the  Ricci curvature on $B_p(R)$ is no less than $(n-1)K$ and $R<\pi/\sqrt{K}$ when $K>0$. Then,
\begin{equation}\label{eq-CY-manifold}
\|\nabla\ln u\|(x)\leq \frac{(2n-3)\sn_K(R)}{2\sn_K\left(\frac{R+r(x)}{2}\right)\sn_K\left(\frac{R-r(x)}{2}\right)}\ \forall x\in B_p(R),
\end{equation}
where $r(x)=r(p,x)$.
\end{thm}
Note that when $n=2$, the estimate \eqref{eq-CY-manifold} becomes \eqref{eq-CY-surface}. However, in the proof of Theorem \ref{thm-CY-manifold}, we need the assumption that $n\geq 3$. Moreover, we don't have any rigidity for \eqref{eq-CY-manifold} because it is not even sharp in leading term comparing to \eqref{eq-LW}.

Finally, by the gradient estimate \eqref{eq-CY-manifold}, we obtain some monotonicity formulas for positive harmonic functions.

\begin{cor}\label{cor-CY-mono-1}
Let $(M^n,g)$ be a complete Riemannian manifold, $p\in M$ and $u$ be a positive harmonic function on $B_p(R)$.  Suppose that the  Ricci curvature on $B_p(R)$ is no less than $(n-1)K$ and $R<\pi/\sqrt{K}$ when $K>0$. Let
$$M_u(r)=\max_{x\in\p B_p(r)}u(x)\mbox{ and }m_u(r)=\min_{x\in\p B_p(r)}u(x).$$ Then, $\left(\frac{\sn_K\left(\frac{R-r}{2}\right)}{\sn_K\left(\frac{R+r}{2}\right)}\right)^{2n-3}M_u(r)$ is decreasing and $\left(\frac{\sn_K\left(\frac{R+r}{2}\right)}{\sn_K\left(\frac{R-r}{2}\right)}\right)^{2n-3}m_u(r)$ is increasing w.r.t. $r\in [0,R)$. In particular,
\begin{equation}\label{eq-Harnack}
\left(\frac{\sn_K\left(\frac{R-r(x)}{2}\right)}{\sn_K\left(\frac{R+r(x)}{2}\right)}\right)^{2n-3}u(p)\leq u(x)\leq\left(\frac{\sn_K\left(\frac{R+r(x)}{2}\right)}{\sn_K\left(\frac{R-r(x)}{2}\right)}\right)^{2n-3} u(p)
\end{equation}
for any $x\in B_p(R)$.
\end{cor}
 When $n=2$ and $K=0$, the Harnack inequality \eqref{eq-Harnack} was obtained in \cite{Xu}.

Finally, we would like to mention that there is a similar story for the Li-Yau gradient estimate (see \cite{LY}) for positive solutions to the heat equation. The classical Li-Yau gradient estimate is sharp for the case that Ricci curvature is nonnegative and is not sharp for the general case that the curvature lower bound is nonzero. It is a natural problem to find sharp Li-Yau gradient estimate for the general case. See \cite{LX,Qi,YZ1,YZ2,YZ3} for some discussions on this problem. Some important monotonicity formulas for positive solutions to heat equations derived from the Li-Yau gradient estimate can be also found in \cite{Ni1,Ni2}.

The rest of the paper is organized as follows. In Section 2, we prove Theorem \ref{thm-CY-conf},  and Theorem \ref{thm-CY-surface}; In Section 3, we prove Theorem \ref{thm-CY-manifold} and Corollary \ref{cor-CY-mono-1}.

\emph{Acknowledgement.} The authors would like to thank Professor Guoyi Xu from Tsinghua University for many helpful discussions. 

\section{Sharp gradient estimate on space forms and surfaces}
In this section, we obtain some sharp Cheng-Yau gradient estimates on space forms and surfaces. These estimates provide test functions for comparison for the general case.

Let's recall the conformal Laplacian operator first. Let $(M^n,g)$ be a Riemannian manifold. The conformal Laplacian operator (see \cite{Pa}) is defined to be
\begin{equation}
\Box_g=-\Delta_g+\frac{n-2}{4(n-1)}R_g
\end{equation}
where $R_g$ is the scalar curvature of $g$. The operator $\Box_g$ is called the conformal Laplacian operator because it satisfies the following identity (see \cite{Pa}):
\begin{equation}\label{eq-conf-inv}
\Box_{e^{2f}g}u=e^{-\frac{n+2}{2}f}\Box_g\left(e^{\frac{n-2}{2}f} u\right),\ \forall f,u\in C^\infty (M).
\end{equation}

We now come to prove Theorem \ref{thm-CY-conf}.
\begin{proof}[Proof of Theorem \ref{thm-CY-conf}]
The case $K=0$ is just Theorem \ref{thm-Xu}. When $K>0$, we only need to prove the theorem for $K=1$, the general case for $K>0$ is then clear by re-scaling. Note that $\mathbb{M}_1^n=\sph^n$. Let $p'$ be the antipodal point of $p$ and $\pi:\sph^{n}\setminus\{p'\}\to \R^n$ the stereoprojection with respect to $p'$. Then, the geodesic ball $B_p(R)$ is mapping to the ball $B(\wt R)\subset \R^n$ with $\wt R=\tan\frac{R}{2}$. Let $\wt g$ be the standard metric on $\R^n$ and $g=(\pi^{-1})^*g_{\sph^n}$. Then
\begin{equation}\label{eq-g}
g=\frac{4}{(1+\|y\|^2)^2}\wt g
\end{equation}
where $\|y\|$ means the Euclidean norm of $y\in \R^n$. By \eqref{eq-conf-inv},
\begin{equation}\label{eq-u}
\wt u(y)=\left(1+\|y\|^2\right)^{\frac{2-n}{2}}u\left(\pi^{-1}(y)\right)
\end{equation}
is a positive harmonic function on $B(\wt R)$. By Theorem \ref{thm-Xu},
\begin{equation}\label{eq-CY-Eu-conf}
\|\wt\nabla \ln \wt u\|(y)\leq \frac{n-1}{\wt R-\|y\|}+\frac{1}{\wt R+\|y\|}, \ \forall y\in B(\wt R).
\end{equation}
Note that
\begin{equation}\label{eq-dist}
\|\pi(x)\|=\tan\frac{r(x)}{2}.
\end{equation}
So, by \eqref{eq-u},
\begin{equation}\label{eq-tilde-u}
\wt u(y)=\cos^{n-2}\left(\frac{r}{2}\right)u\left(\pi^{-1}(y)\right),
\end{equation}
and by \eqref{eq-g}
\begin{equation}\label{eq-grad}
\|\wt \nabla f\|(y)=2\cos^2\left(\frac{r}{2}\right)\|\nabla (f\circ\pi)\|_{g}(\pi^{-1}(y))
\end{equation}
for any $f\in C^\infty(\R^n).$
Then, substituting \eqref{eq-dist}, \eqref{eq-tilde-u} and \eqref{eq-grad} into \eqref{eq-CY-Eu-conf} will give us \eqref{eq-CY-conf} for $K=1$. When the equality of  \eqref{eq-CY-conf} holds for some point in $B_p(R)$, we know that equality of \eqref{eq-CY-Eu-conf} holds for some point in $B(\wt R)$. Then, by Theorem \ref{thm-Xu}, $\wt u$ is a positive multiple of the Poisson kernel w.r.t. some point in $\p B(\wt R)$. Note that the conformal factor is constant on $\p B(\wt R)$. So, $u$ is a positive multiple of the Poisson kernel for $\Box_{\sph^n}$ w.r.t. some point in $\p B_p(R)$. This completes the proof of Theorem \ref{thm-CY-conf} for $K=1$.

When $K<0$, we only need to show the case that $K={-1}$. Note that $\mathbb{M}_{-1}^n$ is the unit ball $\mathbb B^n$ equipped with the metric
\begin{equation}
g=\frac{4}{(1-\|x\|^2)^2}\wt g.
\end{equation}
Because $\mathbb{M}_{-1}^n$ is homogeneous, without loss of generality, we can assume that $p$ is center of $\mathbb B^n$. Note that
\begin{equation}
\|x\|=\tanh\frac{r(x)}{2},\ \forall\ x\in\mathbb B^n
\end{equation}
and
\begin{equation}
\|\wt \nabla f\|(x)=2\cosh^2\left(\frac{r(x)}{2}\right)\|\nabla f\|(x)
\end{equation}
for any $f\in C^\infty(\mathbb B^n)$. Then, the same argument as before will give us \eqref{eq-CY-conf} and its rigidity for $K=-1$. This completes the proof of the theorem.
\end{proof}

We next come to prove Theorem \ref{thm-CY-surface}, a sharp Cheng-Yau gradient estimate on Riemannian surfaces with curvature bounded from below.
\begin{proof}[Proof of Theorem \ref{thm-CY-surface}]
By \cite[Lemma 5.6, P.49]{Li}, one has
\begin{equation}
Q\Delta Q-\|\nabla Q\|^2\geq 2Q^3+2KQ^2
\end{equation}
where $Q=\|\nabla\ln u\|^2$. Let $v=\ln Q$, then
\begin{equation}
\Delta v\geq 2e^v+2K
\end{equation}
on where $Q>0$. Let $$\wt v=F(r(x)):=2\ln\frac{\sn_K(R)}{2\sn_K\left(\frac{R+r(x)}{2}\right)\sn_K\left(\frac{R-r(x)}{2}\right)}.$$
Note that
$$F'(r)=\frac{\sn_K(r)}{\sn_K(\frac{R+r}{2})\sn_K(\frac{R-r}{2})}\geq 0$$
and
$$\Delta r\leq \frac{\cs_K(r)}{\sn_K(r)}$$
on $B_p(R)\setminus \cut(p)$ by the Laplacian comparison, where $\cut(p)$ means the cut-locus of $p$. So, on $B_p(R)\setminus \cut(p)$, we have
\begin{equation*}
\begin{split}
\Delta\wt v=&F''(r)+F'(r)\Delta r\\
\leq&F''(r)+F'(r)\frac{\cs_K(r)}{\sn_K(r)}\\
=&\frac{(F'(r)\sn_K(r))'}{\sn_K(r)}\\
=&\frac{2\cs_K(r)}{\sn_K(\frac{R+r}{2})\sn_K(\frac{R-r}{2})}+\frac{\sn_K^2(r)}{2\sn^2_K(\frac{R+r}{2})\sn^2_K(\frac{R-r}{2})}\\
=&\frac{2\cs_K(r)}{\sn_K(\frac{R+r}{2})\sn_K(\frac{R-r}{2})}+\frac{\sn_K^2(r)-\sn_K^2(R)}{2\sn^2_K(\frac{R+r}{2})\sn^2_K(\frac{R-r}{2})}+2e^{\wt v}\\
=&\frac{2[\cs_K(r)-\cs_K(\frac{R+r}{2})\cs_K(\frac{R-r}{2})]}{\sn_K(\frac{R+r}{2})\sn_K(\frac{R-r}{2})}+2e^{\wt v}\\
=&2 e^{\wt v}+2K
\end{split}
\end{equation*}
where we have used the identity:
$$\cs_K(\alpha-\beta)=\cs_K(\alpha)\cs_K(\beta)+K\sn_K(\alpha)\sn_K(\beta)$$
in the last equality.

The rest of the proof is then the same as the proof of  \cite[(1) \& (2) of Theorem 3.1]{HXY} by using maximum principle and strong maximum principle. This completes the proof of the theorem.
\end{proof}

\section{Gradient estimates and monotonicity formulas}
In this section, by using arguments in the proof of Theorem \ref{thm-CY-surface} using maximum principle,  we prove Theorem \ref{thm-CY-manifold}. Then, as applications of the Cheng-Yau gradient estimates derived, we prove Corollary \ref{cor-CY-mono-1}.

\begin{proof}[Proof of Theorem \ref{thm-CY-manifold}]
By \cite[Lemma 5.6, P.49]{Li},
\begin{equation}\label{eq-CY-n}
Q\Delta Q-\frac{n}{2(n-1)}\|\nabla Q\|^2\geq \frac{2}{n-1}Q^3+2(n-1)KQ-\frac{2(n-2)}{n-1}\vv<Q\nabla Q,\nabla\ln u>
\end{equation}
where $Q=\|\nabla\ln u\|^2$. Then,
\begin{equation}\label{eq-CY-n-1}
\begin{split}
&Q\Delta Q-\frac{n}{2(n-1)}\|\nabla Q\|^2\\
\geq & \frac{2}{n-1}Q^3-\frac{2(n-2)}{n-1}Q^{\frac32}\|\nabla Q\|+2(n-1)KQ^2\\
\geq& \frac{2}{n-1}Q^3-\frac{1}{n-1}\left(2\nu Q^{3}+\frac{(n-2)^2}{2\nu}\|\nabla Q\|^2\right)+2(n-1)KQ^2\\
\end{split}
\end{equation}
where $\nu$ is a constant in $(0,1)$ to be determined. So,
\begin{equation}\label{eq-Q}
Q\Delta Q+q_\nu\|\nabla Q\|^2\geq \frac{2}{n-1}\left(1-\nu\right)Q^3+2(n-1)KQ^2
\end{equation}
where $$q_\nu=\frac{(n-2)^2}{2(n-1)\nu}-\frac{n}{2(n-1)}.$$
Note that $q_\nu+1>0$ when $\nu\in (0,1)$ and $n\geq 3$.  Let
$$v=\|\nabla\ln u\|^{2(1+q_\nu)}=Q^{q_\nu+1}.$$
Then, by \eqref{eq-Q},
\begin{equation}\label{eq-v}
\Delta v\geq\frac{2\left(1-\nu\right)(q_\nu+1)}{n-1}v^{\frac{q_\nu+2}{q_\nu+1}}+2(n-1)(q_\nu+1)Kv.
\end{equation}
Let
$$\wt v=F(r(x)):=\left(\frac{C\cdot\sn_K(R)}{2\sn_K\left(\frac{R+r(x)}{2}\right)\sn_K\left(\frac{R-r(x)}{2}\right)}\right)^{2(q_\nu+1)}$$ with $C$ some positive constant to be determined. Then, by Laplacian comparison:
$$\Delta r\leq (n-1)\frac{\cs_K(r)}{\sn_K(r)}\ \mbox{on }B_p(R)\setminus\cut(p),$$
and that
\begin{equation}
F'(r)=\frac{(q_\nu+1)\sn_K(r)}{\sn_K(\frac{R+r}{2})\sn_K(\frac{R-r}{2})}F(r)\geq 0,
\end{equation}
we have, on $B_p(R)\setminus\cut(p)$,
\begin{equation}
\begin{split}
\Delta \wt v\leq&F''(r)+F'(r)\Delta r\\
\leq&F''(r)+(n-1)F'(r)\frac{\cs_K(r)}{\sn_K(r)}\\
=&n(q_\nu+1)\frac{\cs_K(r)}{\sn_K(\frac{R+r}{2})\sn_K(\frac{R-r}{2})}F+(q_\nu+1)\left(q_\nu+\frac32\right)\frac{\sn_K^2(r)}{\sn_K^2(\frac{R+r}{2})\sn_K^2(\frac{R-r}{2})}F\\
=&n(q_\nu+1)\frac{\cs_K(r)}{\sn_K(\frac{R+r}{2})\sn_K(\frac{R-r}{2})}\wt v+(q_\nu+1)\left(q_\nu+\frac32\right)\frac{\sn_K^2(r)-\sn_K^2(R)}{\sn_K^2(\frac{R+r}{2})\sn_K^2(\frac{R-r}{2})}\wt v\\
&+\frac{2(q_\nu+1)(2q_\nu+3)}{C^2}\wt v^{\frac{q_\nu+2}{q_\nu+1}}\\
=&n(q_\nu+1)\frac{\cs_K(r)}{\sn_K(\frac{R+r}{2})\sn_K(\frac{R-r}{2})}\wt v-(q_\nu+1)\left(2q_\nu+3\right)\frac{2\cs_K(\frac{R+r}{2})\cs_K(\frac{R-r}{2})}{\sn_K(\frac{R+r}{2})\sn_K(\frac{R-r}{2})}\wt v\\
&+\frac{2(q_\nu+1)(2q_\nu+3)}{C^2}\wt v^{\frac{q_\nu+2}{q_\nu+1}}.\\
\end{split}
\end{equation}
Note that
\begin{equation}
2q_\nu+3=\frac{(n-2)^2}{(n-1)\nu}-\frac{n}{n-1}+3\geq\frac{(n-2)^2}{(n-1)}-\frac{n}{n-1}+3=n-1
\end{equation} 
and 
\begin{equation}
n\leq 2(n-1)
\end{equation}
 when $n\geq 2$. So, on $B_p(R)\setminus \cut(p)$, 
\begin{equation*}
\begin{split}
\Delta \tilde v\leq&\frac{2(q_\nu+1)(2q_\nu+3)}{C^2}\wt v^{\frac{q_\nu+2}{q_\nu+1}}+2(n-1)(q_\nu+1)\frac{\cs_K(r)-\cs_K(\frac{R+r}{2})\cs_K(\frac{R-r}{2})}{\sn_K(\frac{R+r}{2})\sn_K(\frac{R-r}{2})}\wt v\\
=&\frac{2(q_\nu+1)(2q_\nu+3)}{C^2}\wt v^{\frac{q_\nu+2}{q_\nu+1}}+2(n-1)(q_\nu+1)K\wt v\\
=&\frac{2\left(1-\nu\right)(q_\nu+1)}{n-1}\wt v^{\frac{q_\nu+2}{q_\nu+1}}+2(n-1)(q_\nu+1)K\wt v,
\end{split}
\end{equation*}
when
\begin{equation}\label{eq-C-Eu}
C^2=\frac{(n-1)(2q_\nu+3)}{1-\nu}=\frac{(n-2)^2+(2n-3)\nu}{(1-\nu)\nu}.
\end{equation}
Then, the same as in the proof of \cite[(1) of Theorem 3.1]{HXY} by using maximum principle and Calabi's trick \cite{Ca}, one has
\begin{equation}\label{eq-comparison}
v\leq \wt v.
\end{equation}
Finally, note that the minimum value of
\begin{equation*}
f(\nu)=\frac{(n-2)^2+(2n-3)\nu}{(1-\nu)\nu}
\end{equation*}
with $\nu\in (0,1)$ is $(2n-3)^2$ with minimum point $\nu_m=\frac{n-2}{2n-3}$. Then, setting $\nu=\nu_m$ and substituting it into \eqref{eq-comparison}, we complete the proof of the theorem.
\end{proof}

Finally, we prove Corollary \ref{cor-CY-mono-1}.
\begin{proof}[Proof of Corollary \ref{cor-CY-mono-1}]
Let $\gamma$ be a normal geodesic starting at $p$. By Theorem \ref{thm-CY-surface} and Theorem \ref{thm-CY-manifold}, for any $0<r_1<r_2<R$,
\begin{equation*}
\begin{split}
&\left|\ln u(\gamma(r_2))-\ln u(\gamma(r_1))\right|\\
=& \left|\int_{r_1}^{r_2}d\ln u(\gamma(s))\right|\\
=&\left|\int_{r_1}^{r_2}\vv<\nabla \ln u(\gamma(s)),\gamma'(s)>ds\right|\\
\leq& \int_{r_1}^{r_2}\|(\nabla\ln u)(\gamma(s))\|ds\\
\leq&\int_{r_1}^{r_2}\frac{(2n-3)\sn_K(R)}{2\sn_K\left(\frac{R+r(\gamma(s))}{2}\right)\sn_K\left(\frac{R-r(\gamma(s))}{2}\right)}ds\\
=&\int_{r_1}^{r_2}\frac{(2n-3)\sn_K(R)}{2\sn_K(\frac{R+s}{2})\sn_K(\frac{R-s}{2})}ds\\
=&(2n-3)\ln\frac{\sn_K\left(\frac{R+s}{2}\right)}{\sn_K\left(\frac{R-s}{2}\right)}\Bigg|^{r_2}_{r_1}.
\end{split}
\end{equation*}
So,
\begin{equation}\label{eq-decreasing}
  \ln u(\gamma(r_2))-(2n-3)\ln\frac{\sn_K\left(\frac{R+r_2}{2}\right)}{\sn_K\left(\frac{R-r_2}{2}\right)}\leq \ln u(\gamma(r_1))-(2n-3)\ln\frac{\sn_K\left(\frac{R+r_1}{2}\right)}{\sn_K\left(\frac{R-r_1}{2}\right)}
\end{equation}
which implies that $\left(\frac{\sn_K\left(\frac{R-r}{2}\right)}{\sn_K\left(\frac{R+r}{2}\right)}\right)^{2n-3}M_u(r)$ is decreasing, and
\begin{equation}\label{eq-increasing}
\ln u(\gamma(r_1))+(2n-3)\ln\frac{\sn_K\left(\frac{R+r_1}{2}\right)}{\sn_K\left(\frac{R-r_1}{2}\right)}\leq\ln u(\gamma(r_2))+(2n-3)\ln\frac{\sn_K\left(\frac{R+r_2}{2}\right)}{\sn_K\left(\frac{R-r_2}{2}\right)}
\end{equation}
which implies that $\left(\frac{\sn_K\left(\frac{R+r}{2}\right)}{\sn_K\left(\frac{R-r}{2}\right)}\right)^{2n-3}m_u(r)$ is increasing. More precisely, let $x_0\in \p B_p(r_2)$ be such that $u(x_0)=\max_{x\in \p B_p(r_2)}u(x)$ and $\gamma$ be a normal minimal geodesic joining $p$ to $x_0$.
Then, by \eqref{eq-decreasing},
\begin{equation*}
\begin{split}
M_u(r_2)\left(\frac{\sn_K\left(\frac{R-r_2}{2}\right)}{\sn_K\left(\frac{R+r_2}{2}\right)}\right)^{2n-3}=&u(\gamma(r_2))\left(\frac{\sn_K\left(\frac{R-r_2}{2}\right)}{\sn_K\left(\frac{R+r_2}{2}\right)}\right)^{2n-3}\\
\leq& u(\gamma(r_1))\left(\frac{\sn_K\left(\frac{R-r_1}{2}\right)}{\sn_K\left(\frac{R+r_1}{2}\right)}\right)^{2n-3}\\
\leq& M_u(r_1)\left(\frac{\sn_K\left(\frac{R-r_1}{2}\right)}{\sn_K\left(\frac{R+r_1}{2}\right)}\right)^{2n-3}.\\
\end{split}
\end{equation*}
So, $\left(\frac{\sn_K\left(\frac{R-r}{2}\right)}{\sn_K\left(\frac{R+r}{2}\right)}\right)^{2n-3}M_u(r)$ is decreasing. By a similar argument using \eqref{eq-increasing}, $\left(\frac{\sn_K\left(\frac{R+r}{2}\right)}{\sn_K\left(\frac{R-r}{2}\right)}\right)^{2n-3}m_u(r)$ is increasing. This completes the proof of the corollary. 
\end{proof}

\end{document}